\documentclass[12pt,twoside]{article}
\usepackage{a4}
\usepackage{jgg}
\usepackage{amsfonts,amssymb,amsmath,amsthm}

\usepackage{dsfont}

\newtheoremstyle{mystyle}
  {3pt}
  {3pt}
  {}
  {}
  {\bfseries}
  {:}
  {0.5 em}
  {}

\theoremstyle{mystyle}
\newtheorem{definition}{Definition}

\newtheorem{remark}{Remark}

\newtheorem{lemma}{Lemma}
\newtheorem{example}{Example}

\usepackage[latin1]{inputenc}	
\usepackage{placeins}

\usepackage{hyperref}
\hypersetup{
    colorlinks=false,
    pdfborder={0 0 0},
}

\newcommand{\clifford}{\mathcal{C}\ell}
\newcommand{\pgl}{\mathrm{PGL}}
\newcommand{\se}{\mathrm{SE}}
\newcommand{\so}{\mathrm{SO}}

\begin{document}
\begin{JGGarticle}
{Kinematic Mappings for Cayley-Klein Geometries via Clifford Algebras}
{}
{Daniel Klawitter, Markus Hagemann}
{\JGGaddress{
Dresden University of Technology, Germany}
}
\begin{JGGabstract}\\
This paper unifies the concept of kinematic mappings by using geometric algebras. We present a method for constructing kinematic mappings for certain Cayley-Klein geometries. These geometries are described in an algebraic setting by the homogeneous Clifford algebra model. Displacements correspond to Spin group elements. After that Spin group elements are mapped to a kinematic image space. Especially for the group of planar Euclidean displacements $\se(2)$ the result is the kinematic mapping of Blaschke and Gr\"unwald. For the group of spatial Euclidean displacements  $\se(3)$ the result is Study's mapping. Furthermore, we classify kinematic mappings for Cayley-Klein spaces of dimension $2$ and $3$.\\
[1mm]{\em Key Words:}
Kinematic mapping, Cayley-Klein geometry, Study's quadric, Clifford algebra.
\end{JGGabstract}

\section*{Introduction}
To make this paper self-contained we shall give a brief introduction to Clifford respectively geometric algebras. This paper is organized as follows: In Section 1 we introduce the concept of geometric algebras. Therefore, Clifford algebras and some of their properties are introduced. Furthermore, Section 1 deals with the Pin and Spin group of a Clifford algebra.
After that in Section 2, we recall Cayley-Klein spaces and show how to describe certain Cayley-Klein spaces within the homogeneous Clifford algebra model. This construction is based on the work of Charles Gunn, cf. \cite{gunn:onthehomogeneousmodel,gunn:kinematics}, and is accomplished in detail for the $3$-dimensional Euclidean space. Section 3 deals with kinematic mappings. The kinematic mapping of Study and the mapping of Blaschke and Gr\"unwald are presented. Thereafter, we describe how to construct these mappings and other in an unified method in Section 4. Furthermore, the matrices of the collineations in the image and pre-image space are derived. Again we do this in detail for the Euclidean spaces of dimension $2$ and $3$. Section 5 gives an overview of possible kinematic mappings for Cayley-Klein spaces of dimension $2$ and $3$. Moreover, the mapping for the $4$-dimensional Euclidean space is presented.

\section{Geometric Algebras}
Geometric algebras are special Clifford algebras over the field of real numbers. They are associative algebras that generalize the complex numbers, Hamilton's quaternions and biquaternions. Geometric algebras, abbreviated by GA, find application in computer graphics \cite{dorst:geometricalgebra}, robotics \cite{gunn:kinematics,selig:geometricfundamentalsofrobotics,mccarty:theoreticalkinematics}, physics \cite{hestenes:cliffordalgebra} and a lot of other disciplines. Here we focus on their application in kinematics of some Cayley-Klein geometries. The great advantage of GA is that geometric entities such as points, lines and planes can be described as algebra elements. Furthermore, transformations can also be described as special algebra elements and the effect of the transformations applied to a geometric entity is realized by an algebra operation. A complete treatise of this topic is outside the scope of this paper, but we try to give enough references for the interested reader.

\subsection{Definition of a geometric algebra}\label{section1.1}
We start with a finite dimensional real vector space $V=\mathds{R}^n$ equipped with a quadratic form $b:V\rightarrow\mathds{R}$. The pair $(V,b)$ is called a \emph{quadratic space}. The symmetric matrix corresponding to the quadratic form $b$ is denoted by $B_{ij}$ with $1\leq i,j\leq n$. Therefore, $b(x_i,x_j)=B_{ij}$ for some basis vectors $x_i$ and $x_j$. The algebra is defined by the relations
\begin{equation}\label{label.2.1}
x_ix_j + x_jx_i=2B_{ij},\quad 1\leq i,j\leq n.
\end{equation}
Usually the corresponding Clifford algebra is denoted by $\clifford(V,b)$. We shall use another notation according to this algebra. With Sylvester's law of inertia we can always find a basis $\left\lbrace e_1,\dots,e_n \right\rbrace $ of $V$ such that $e_i^2$ is either $1,-1$ or $0$. The number of basis vectors that square to $(1,-1,0)$ is denoted by the \emph{signature} $(p,q,r)$. Therefore, we will use the notation $\clifford_{(p,q,r)}$ instead of $\clifford(V,b)$. If $r\neq 0$ we call the geometric algebra \emph{degenerated}. Furthermore, the relations \eqref{label.2.1} become
 \begin{equation}\label{label.2.2}
 e_ie_j+e_je_i=0,\quad i\neq j.
 \end{equation}
For details the interested reader is referred to \cite{selig:geometricfundamentalsofrobotics}. Note that Eq. \eqref{label.2.2} also shows that two basis elements anti-commute, {\it i.e.},
\begin{equation*}
e_ie_j=-e_je_i,\quad i\neq j.
\end{equation*}
In the remainder of this paper we shall abbreviate the geometric product of basis elements with lists
\[e_{12\dots k}:=e_1e_2\dots e_k, \mbox{ with $0\leq k\leq n$}.\]
With this notation scalars, bivectors, trivectors and pseudo scalars are represented by $ae_0$,
,$\sum\limits_{\substack{i,j=1\\i<j}}^n e_{ij}a_{ij}$, $\sum\limits_{\substack{i,j,k=1\\i<j<k}}^n e_{ijk}a_{ijk}$ and $ae_{12\dots n}$, where all occurring coefficients are real numbers.
\begin{example}\label{example1}
In order to bring light into this definition, we present the quaternions as elements of a Clifford algebra. Therefore, we construct the Clifford algebra $\clifford_{(0,2,0)}$. This means we have two basis vectors, that square to $-1$. The field $\mathds{R}$ is embedded with the basis element $e_0$ and an element of $\clifford_{(0,2,0)}$ can be written as:
\[\mathfrak{a}=a_0e_0+a_1e_1+a_2e_2+a_{12}e_{12}.\]
We write Clifford algebra elements as small fractal letters. The basis\linebreak
$\left\lbrace e_0,e_1,e_2,e_{12}\right\rbrace $ of $\clifford_{(0,2,0)}$ is called a \emph{standard basis} of this geometric algebra.
As we already know from the definition $e_1^2=e_2^2=-1$. The square of $e_{12}$ evaluates to
\[e_{12}^2=e_{12}e_{12}=-e_{12}e_{21}=-e_{1221}=e_{11}=-1.\]
If we now identify $e_1$ with $\mathbf{i}$, $e_2$ with $\mathbf{j}$, and $e_{12}$ with $\mathbf{k}$, we get an isomorphism between $\clifford_{(0,2,0)}$ and $\mathds{H}$. Of course the multiplication rules for the quaternion units $\mathbf{i},\mathbf{j}$ and $\mathbf{k}$ have to be  verified. This is left to the reader as an exercise.
\end{example}
\noindent
For every Clifford algebra over an $n$-dimensional vector space a general element is the sum of scalars, vectors, bivectors, trivectors up to pseudo scalars, it is called a \emph{multivector} and reads
\[\mathfrak{m}=a_0e_0+a_1e_1+a_ne_n+a_{12}e_{12}+\dots +a_{n(n-1)}e_{n(n-1)}+\dots +a_{1\dots n}e_{1\dots n}.\]

\subsection{Properties of Clifford Algebras}
\paragraph*{A Clifford algebra is a graded algebra:$\,$}
Every basis element $e_{\alpha\beta\dots\gamma}$ can be transformed to a basis element of the form $e_{ij\dots k}$, where $i<j<k$. Each swap of two elements causes a multiplication by $-1$. If we arrive at $e_{ii}$ we can insert $1,-1$ or $0$ as prescribed by the signature $(p,q,r)$. The $2^n$ monomials
\[e_{i_1 i_2 \dots i_k},\quad 0\leq k\leq n\]
form the standard basis of the Clifford algebra. An algebra element that is the product of invertible grade-1 elements is called a \emph{versor}. The Clifford algebra $\clifford_{(p,q,r)}$ is the direct sum $\bigoplus\limits_{i=0}^n{\bigwedge}^i V$ of all exterior products ${\bigwedge}^i V$ of $V$ of any grade $0\leq i\leq n$ where $e_{k_1\dots k_i},k_1<\dots<k_i$ form a basis of ${\bigwedge}^i V$.
Especially ${\bigwedge}^0 V$ is the scalar part $\mathds{R}$ and ${\bigwedge}^1 V$ represents the vector space $V$. Bivectors, trivectors and \emph{pseudo scalars} are elements from ${\bigwedge}^2 V $, ${\bigwedge}^3 V $ and ${\bigwedge}^n V $, respectively. The dimension of each subspace ${\bigwedge}^i V$ is $\binom{n}{i}$. Thus, the dimension of the Clifford algebra is $\sum_{i=0}^n(\dim {\bigwedge}^i V)=\sum_{i=0}^n\binom{n}{i}=2^n$. A Clifford algebra is called \emph{universal} if $\dim \clifford_{(p,q,r)}=2^n,\, n=p+q+r$. For every quadratic space $(V,b)$ there exists an universal Clifford algebra, see \cite{garling:cliffordalgebras}. Therefore, we will restrict ourself to the standard basis of the corresponding universal algebra. Furthermore, the Clifford algebra $\clifford_{(p,q,r)}$ can be decomposed in an even and an odd part
\begin{equation*}
\clifford_{(p,q,r)}=\clifford^+_{(p,q,r)}\oplus \clifford^-_{(p,q,r)}=\bigoplus_{\substack{i=0\\i\text{ mod $2=0$} }}^n{\bigwedge}^iV\oplus \bigoplus_{\substack{i=0\\i \text{ mod $2=1$}}}^n{\bigwedge}^iV.
\end{equation*}
Note that the even part $\clifford^+_{(p,q,r)}$ is a subalgebra, because the product of two even graded monomials must be even graded and the generators cancel only in pairs. The dimension of the even subalgebra is $2^{n-1}$. Furthermore we have the isomorphism
\begin{equation*}
\clifford_{(p,q,r)}\cong\clifford^+_{(p,q+1,r)}.
\end{equation*}
For details see \cite{porteous:cliffordalgebrasandtheclassicalgroups}. To make this isomorphism more clear we give another example.
\begin{example}\label{example2}
As we already know from example \ref{example1} the Clifford algebra $\clifford_{(0,2,0)}$ is isomorphic to $\mathds{H}$. The even part of this algebra is generated by $\left\lbrace e_0,e_{12}\right\rbrace $ and this is isomorphic to $\mathds{C}$. Furthermore, the Clifford algebra $\clifford_{0,1,0}$ is the algebra with one generator squaring to $-1$. Therefore, $\clifford_{(0,1,0)}\cong \mathds{C}$. All in all we have $\clifford^+_{0,2,0}\cong\clifford_{0,1,0}$.
\end{example}

\subsection{Clifford algebra automorphisms}
\paragraph{The conjugation:}
Every Clifford algebra possesses an \emph{anti-involution}, {\em i.e.}, an automorphism of the algebra that reverses the order of products. We will follow \cite{selig:geometricfundamentalsofrobotics} and denote the conjugation by an asterisk. The effect on any generator is given by $e_i^\ast=-e_i$. On scalars it has no effect. If we extend the conjugation by using linearity and the anti-involution property to arbitrary algebra elements, we get the formula
\begin{equation*}
(e_{i_1}e_{i_2}\dots e_{i_k})^\ast=(-1)^ke_{i_k}\dots e_{i_2}e_{i_1},\quad 0\leq i_1<i_2<\dots <i_k\leq n.
\end{equation*}
\begin{example}\label{example3}
We apply the conjugation to an algebra element of the Clifford algebra $\clifford_{(0,2,0)}\cong\mathds{H}$. The standard basis is given by $\left\lbrace e_0,e_1,e_2,e_{12}\right\rbrace $. The conjugation for the grade-1 basis elements $e_1$ and $e_2$ is given via definition by $e_1^\ast=-e_1$ and $e_2^\ast=-e_2$. The conjugation has no effect on scalars, so $e_0^\ast=e_0$. It remains to extend these observations to the grade-2 element $e_{12}$. We use the anti-involution property and the anti commutativity
\[e_{12}^\ast=(e_1e_2)^\ast=e_2^\ast e_1^\ast=(-1)^2e_2e_1=e_{21}=-e_{12}.\]
Hence, the conjugation of quaternions fits into the concept.
\end{example}
\noindent 
A general grade-1 element in $\clifford_{(p,q,r)}$ has the form
\[\mathfrak{v}=x_1e_2+x_2e_2+\dots+x_ne_n,\quad n=p+q+r.\]
The product with its conjugate element is given by:
\begin{equation}\label{label.2.8}
\mathfrak{vv}^\ast=-x_1^2-x_2^2-\dots-x_p^2+x_{p+1}^2+\dots+x_{p+q}^2=-b(x,x).
\end{equation}
If we identify $V$ with ${\bigwedge}^1 V$ Eq. \eqref{label.2.8} describes the negative of the quadratic form $b(x,x)$ with $x=(x_1,\dots,x_n)^T\in V$.
\paragraph{The main involution:} Another automorphism of a Clifford algebra is the \emph{main involution}. It is denoted by $\alpha$ and defined by
\begin{equation*}
\alpha(e_{i_1}e_{i_2}\dots e_{i_k})=(-1)^k e_{i_1}e_{i_2}\dots e_{i_k},\quad 0\leq i_1<i_2<\dots <i_k\leq n.
\end{equation*}
The main involution has no effect on the even subalgebra and it commutes with the conjugation. This means for an arbitrary algebra element $\mathfrak{a}\in\clifford_{(p,q,r)}$, the equation $\alpha(\mathfrak{a}^\ast)=\alpha(\mathfrak{a})^\ast$ holds.
\subsection{Clifford Algebra Products}
\paragraph{The inner product:}
The scalar product of two vectors, {\em i.e.}, grade-1 elements can be written in terms of the geometric product
\begin{equation}\label{label.2.10}
\mathfrak{a}\cdot \mathfrak{b}:=\frac{1}{2}(\mathfrak{ab}+\mathfrak{ba}).
\end{equation}
This can also be generalized to multivectors, see for example \cite{hestenes:cliffordalgebra}. For our purposes it is sufficient to define this product for vectors.
\paragraph{The outer product:}
The outer respectively exterior product for two vectors is given by
\begin{equation}\label{label.2.11}
\mathfrak{a}\wedge \mathfrak{b}:=\frac{1}{2}(\mathfrak{ab}-\mathfrak{ba}).
\end{equation}
\paragraph{The geometric product:}
With the inner product \eqref{label.2.10} and the outer product \eqref{label.2.11} the geometric product of vectors can be written in the following form
\begin{equation*}
\mathfrak{ab}=\mathfrak{a}\cdot \mathfrak{b}+\mathfrak{a}\wedge \mathfrak{b}.
\end{equation*}
\subsection{Pin and Spin-groups}
Not every element of a Clifford algebra needs to have an inverse element. In general there are zero divisors, this means for two non zero elements $\mathfrak{a},\mathfrak{b}\in\clifford_{(p,q,r)}$ the product $\mathfrak{ab}$ is zero.
\begin{example}
A simple example for a zero divisor is given by an element $\mathfrak{a}=ae_1+ae_2\in\clifford_{(1,1,0)}$. This Clifford algebra has the standard basis $\left\lbrace e_0,e_1,e_1,e_{12}\right\rbrace $ with $e_1^2=1,\,e_2^2=-1$ and $e_{12}^2=1$. We calculate:
\[\mathfrak{a}^2=(ae_1+ae_2)^2=a^2e_1^2+a^2e_{12}+a^2e_{21}+a^2e_2^2=0.\]
Note that $\clifford_{(1,1,0)}\cong \mathds{R}^{2\times 2}$.
\end{example}
\noindent
The inverse element for a versor is defined by $\mathfrak{a}^{-1}:=\frac{\mathfrak{a}^\ast}{N(\mathfrak{a})}$ with $N(\mathfrak{a}):=\mathfrak{aa}^\ast$. The map $N:\clifford_{(p,q,r)}\rightarrow\clifford_{(p,q,r)}$ is called the \emph{norm} of the Clifford algebra. For general multivectors $\mathfrak{m}\in\clifford_{(p,q,r)}$, that are no zero divisors, inverse elements exist and are defined through the relation $\mathfrak{mm}^{-1}=\mathfrak{m}^{-1}\mathfrak{m}=1$, but the determination is more difficult. This will be discussed in Section \ref{matrixsection}.
\subsection{Clifford Group}
Invertible elements are called \emph{units}. We shall denote the set of units of a Clifford algebra by $\clifford_{(p,q,r)}^\times$. With respect to the geometric product the units form a group. 
For a given Clifford algebra the \emph{Clifford group} is defined by
\begin{equation*}
\Gamma(\clifford_{(p,q,r)}):=\left\{\mathfrak{g}\in\clifford_{(p,q,r)}^\times\left|\alpha(\mathfrak{g})\mathfrak{vg}^{-1}\in {\bigwedge}^1 V \,\mbox{ for all } \mathfrak{v}\in {\bigwedge}^1 V\right.\right\}.
\end{equation*}
It is not obvious that $\Gamma(\clifford_{(p,q,r)})$ is a group. For a proof we refer to \cite{gallier:cliffordalgebrascliffordgroups}. Note that the condition $N(\mathfrak{g})=1$ guarantees that $\mathfrak{g}$ is an unit.
\paragraph{The Pin group:} The Pin group is defined by
\begin{equation*}
\mbox{Pin}_{(p,q,r)}:=\left\{\mathfrak{g}\in\clifford_{(p,q,r)}\left|N(\mathfrak{g}) = 1\mbox{ and }\alpha(\mathfrak{g})\mathfrak{vg}^\ast\in {\bigwedge}^1 V\, \mbox{ for all } \mathfrak{v}\in {\bigwedge}^1 V\right.\right\}.
\end{equation*}
The action of the so called \emph{sandwich operator} $\alpha(\mathfrak{g})\mathfrak{vg}^\ast$ with $\mathfrak{g}\in\mbox{Pin}_{(p,q,r)}$ applied to vectors does not change the scalar product of two vectors. This can be verified easily by direct calculation. Let $\mathfrak{a,b}\in{\bigwedge}^1 V$ and $\mathfrak{g}\in\mbox{Pin}(p,q,r)$. 
Furthermore, let $\mathfrak{a}'=\alpha(\mathfrak{g})\mathfrak{ag}^\ast$ and $\mathfrak{b}'=\alpha(\mathfrak{g})\mathfrak{bg}^\ast$. The scalar product can be expressed in terms of the geometric product, see Eq. \eqref{label.2.10}. We use this equation for the transformed vectors $\mathfrak{a}'$ and $\mathfrak{b}'$ and find
\begin{align*}
\mathfrak{a}'\cdot \mathfrak{b}'&=\frac{1}{2}\left( \alpha(\mathfrak{g})\mathfrak{ag}^\ast \alpha(\mathfrak{g})\mathfrak{bg}^\ast + \alpha(\mathfrak{g})\mathfrak{bg}^\ast \alpha(\mathfrak{g})\mathfrak{ag}^\ast\right)\\
		 &=\frac{1}{2}\alpha(\mathfrak{g})\left[ \mathfrak{ag}^\ast \alpha(\mathfrak{g})\mathfrak{b} + \alpha(\mathfrak{g})\mathfrak{bg}^\ast \alpha(\mathfrak{g})\mathfrak{a}\right]\mathfrak{g}^\ast.
\intertext{Since $\alpha(\mathfrak{g})$ is either $\mathfrak{g}$ or $-\mathfrak{g}$, we can conclude that $\mathfrak{g}^\ast \alpha(\mathfrak{g})=\alpha(\mathfrak{g})\mathfrak{g}^\ast$ is either $1$ or $-1$. Therefore, the term in square brackets is a scalar and consequently}
\mathfrak{a}'\cdot \mathfrak{b}'&=\frac{1}{2}\alpha(\mathfrak{g})\mathfrak{g}^\ast\left[ \mathfrak{g}^\ast \alpha(\mathfrak{g})(\mathfrak{ab} + \mathfrak{ab})\right].
\intertext{The product $\mathfrak{g}^\ast \alpha(\mathfrak{g})=\alpha(\mathfrak{g})\mathfrak{g}^\ast=\pm 1$. Thus, we get}
\mathfrak{a}'\cdot \mathfrak{b}' &=\frac{1}{2}(\alpha(\mathfrak{g})\mathfrak{g}^\ast)(\mathfrak{g}^\ast\alpha(\mathfrak{g}))\left(\mathfrak{ab} + \mathfrak{ab}\right)
		  =\frac{1}{2}\left(\mathfrak{ab} + \mathfrak{ab}\right)
		  =\mathfrak{a}\cdot \mathfrak{b}.
\end{align*}
Linear transformations of ${\bigwedge}^1 V$ that preserve distances and angles, {\em i.e.}, the scalar product, are elements of the orthogonal group $O(p,q,r)$. In fact $\mbox{Pin}(p,q,r)$ is a double cover of the orthogonal group $O(p,q,r)$.
\begin{remark}
The action $\mathfrak{a}\mapsto \alpha(\mathfrak{v})\mathfrak{av}^\ast$ for $\mathfrak{v}\in\mbox{Pin}(p,q,r)\cap {\bigwedge}^1 V$ is a reflection in the hyperplane perpendicular to $\mathfrak{v}$ with respect to $b$. The composition of reflections is again an element of the Pin group.
\end{remark}
\paragraph{The Spin group:} The second important subgroup of the group of units is the \emph{Spin group} defined via
\begin{align*}
\mbox{Spin}_{(p,q,r)}:&=\left\{\mathfrak{g}\in\clifford^+_{(p,q,r)}\left|N(\mathfrak{g}) = 1\mbox{ and }\mathfrak{gvg}^\ast\in {\bigwedge}^1 V\, \mbox{ for all } \mathfrak{v}\in {\bigwedge}^1V\right.\right\}\\
				 &=\mbox{Pin}(n)\cap\clifford^+_{(p,q,r)}.
\end{align*}
Note that the main involution $\alpha$ hast no effect on the even subalgebra and therefore can be forgotten.
\begin{remark}
The Spin group is a subgroup of the Pin group. Therefore, the scalar product of vectors is preserved under the action of the Spin group. Spin group elements are generated by pairs of reflections, so they are rotations.
\end{remark}
\noindent
The Spin group for degenerated Clifford algebras as in the case of Euclidean geometry are semi-direct products of Spin groups for the non-degenerated part and an additive matrix group, see \cite{dereli:degenratespingroups}.
\subsection{Matrix representation of Clifford algebras}\label{matrixsection}
Following \cite{mccarty:theoreticalkinematics}, any multivector from a Clifford algebra over an $n$-dimensional vector space can be represented by a vector $v\in \mathds{R}^{2^n}$. Then the geometric product of two algebra elements $\mathfrak{ab=c}$ can be written as product of a vector and a matrix.
\begin{equation}\label{matrixrepresentation}
C=\left[A^+ \right] B,\quad C=\left[B^- \right] A, \quad \left[A^+ \right], \left[B^- \right]\in \mathds{R}^{2^n\times 2^n}.
\end{equation}
The columns of the matrix $\left[A^+ \right]$ are defined by the products 
\[\mathfrak{a}e_0,\mathfrak{a}e_1,\dots ,\mathfrak{a}e_n,\mathfrak{a}e_{12},\dots \mathfrak{a}e_{(n-1)n},\dots ,\mathfrak{a}e_{1\dots n}.\] Note that these columns are ordered from right to left to match the chosen ordering of the vector representation for the Clifford algebra. To obtain the columns of the matrix $\left[B^- \right]$ we have to multiply $\mathfrak{b}$ from the left side with each basis element \[e_0\mathfrak{b},e_1\mathfrak{b},\dots ,e_n\mathfrak{b},e_{12}\mathfrak{b},\dots e_{(n-1)n}\mathfrak{b},\dots ,e_{1\dots n}\mathfrak{b}\] again with the same ordering.
\begin{remark}
With the matrix representation it is possible to calculate the inverse element for an arbitrary multivector. Just express the geometric product as matrix vector product.
If the $2^n\times 2^n$ matrix is invertible, we see that the inverse algebra element corresponds to this matrix. For increasing vector space dimension this calculation can be extremely expensive.
\end{remark}
\section{Cayley-Klein geometries and the homogeneous model}
Cayley-Klein spaces are metric spaces constructed within an $n$-dimensional projective space $\mathds{P}^n(\mathds{R})$ with a distinguished quadratic hypersurface given by a quadratic form
\[x^T B x = 0,\]
where $B$ is a $n+1\times n+1$ symmetric matrix and $x\in \mathds{P}^n$. This idea goes back to A. Cayley and F. Klein. Comprehensive work on this field of geometry can be found in \cite{kowol,onishchik} and \cite{struve}. With Silvester's law of inertia we can always find a diagonal matrix corresponding to the quadratic hypersurface, see Section \ref{section1.1}. This diagonal form is called the normal form of the quadric hypersurface and corresponds to the signature of the Clifford algebra that we will use later on. The Cayley-Klein construction provides models for the Euclidean, hyperbolic, elliptic, and many other geometries. An exhaustive treatise of this topic can be found in \cite{giering:vorlesungenueberhoeheregeometrie}. The definition of a Cayley-Klein space that we will use can also be found therein.
\begin{definition}
Let $\mathds{P}^n(\mathds{R})$ be the $n$-dimensional projective space, then $\mathcal{F}\subset\mathds{P}^n(\mathds{C})$ is defined by the chain
\begin{equation*}
\mathcal{F}:Q_{r_0,q_0}^n\supset A^{n_1}\supset Q_{r_1,q_1}^{n_1-1}\supset \dots \supset A^{n_\rho}\supset Q_{r_\rho,q_\rho}^{n_\rho-1}.
\end{equation*}
Here $r_i$ is the rank, $q_i$ the index and $n_i-1$ the dimension of the quadratic variety $Q_{r_i,q_i}^{n_i-1}$ that is for $i<\rho$ a singular quadric and for $i=\rho$ a regular one. Furthermore, the identity $n+1=r_0+r_1+r_\rho$ holds. $\mathds{P}^n$ with absolute figure $\mathcal{F}$ is called a \emph{Cayley-Klein space}. The points contained in $\mathcal{F}$ are called \emph{ideal} points and the points in $\mathds{P}^n\backslash\mathcal{F}$ \emph{proper} points. $Q_{r_i,q_i}^{n_i-1}$ is called \emph{absolute quadric} or \emph{absolute cone}. $A^{n_i}$ is called \emph{absolute $n_i$-subspace}. A Cayley-Klein space where the set of points $W$ is removed is called \emph{sliced} along $W$.
\end{definition}
\begin{remark}
Points that are neither ideal nor proper are called \emph{improper} points. These occur for example in the case of the hyperbolic geometry.
\end{remark}
\noindent
The construction of a Clifford algebra over a projective space results in the so called \emph{homogeneous model}.
\subsection{Cayley-Klein spaces}
At first we present an example for a planar Cayley-Klein space. This means we construct the homogeneous model over $\mathds{R}^3$ with coordinates $x_0,x_1,x_2$ as vector space model for $\mathds{P}^2(\mathds{R})$ and an absolute figure. We start with the \emph{Euclidean plane}. For this purpose an absolute figure of the form
\begin{equation}\label{label.2.17}
\mathcal{F}:Q_{1,0}^{1}\supset A^1\supset Q_{2,0}^0
\end{equation}
is prescribed. The quadric can be written as $Q_{1,0}^{1}:x_0^2=0$. It is singular, thus cone shaped and has the vertex space $A^1:x_0=0$ which is a line. The regular quadric at the end of the chain  is given by $x_1^2+x_2^2=0$. The projective automorphisms of $\mathcal{F}$ form a subgroup $\pgl(\mathds{P}^2,\mathcal{F})\subset \pgl(\mathds{P}^2)$ and thus they constitute the group of motions in this Cayley-Klein geometry. In the Euclidean case this group is the group $\se(2)$ of planar Euclidean displacements. Now we construct the homogeneous model for this Cayley-Klein space. Therefore, we take $\mathds{R}^3$ as vector space and a quadratic form with signature $(2,0,1)$ to construct the Clifford algebra $\clifford_{(2,0,1)}$. Points $P=(x_0,x_1,x_2)^T \in\mathds{P}^2(\mathds{R})$ of the Euclidean plane are described in the ${\bigwedge}^2 V$ subspace of the algebra via $\mathfrak{p}=x_0e_{12}+x_1e_{13}+x_2e_{23}$. Note that the homogeneous factor of the point is contained in the $e_{12}$ component. The scalar product is calculated by $\mathfrak{p}\cdot \mathfrak{p}=-x_0^2$. Hence, we have the possibility to norm with operations in the algebra. Thus, we find Euclidean geometry considered as a Cayley-Klein geometry within this model. The group $\pgl(\mathds{P}^2,\mathcal{F})$ can be found in the Clifford algebra as the group $\mbox{Spin}_{(2,0,1)}$. In fact the Spin group is a double cover of the group $\pgl(\mathds{P}^2,\mathcal{F})$.
\paragraph{3-dimensional Euclidean space:} The $3$-dimensional Euclidean space is constructed in the same way. We start with $\mathds{P}^3(\mathds{R})$ and the absolute figure
\begin{equation}\label{3deuclidean}
\mathcal{F}:Q_{1,0}^{2}\supset A^2\supset Q_{2,0}^1,
\end{equation}
wherein the first singular quadric in the chain is given by $Q_{1,0}^{2}:x_0^2=0$ and its vertex is given by $A^2:x_0=0$. Contained in $A^2$ we have the so called \emph{absolute circle} $Q_{2,0}^1:x_1^2+x_2^2+x_3^2=0$. The subgroup $\pgl(\mathds{P}^3,\mathcal{F})\subset \pgl(\mathds{P}^3)$ that fixes this absolute figure can be identified as the group of spatial Euclidean displacements $\se(3)$. We aim at a representation of this Cayley-Klein geometry with a homogeneous Clifford algebra model. Therefore, we take $\mathds{R}^4$ as vector space model for $\mathds{P}^3(\mathds{R})$ together with a quadratic form of signature $(3,0,1)$ to obtain the Clifford algebra $\clifford_{(3,0,1)}$. Again the group $\mbox{Spin}_{(3,0,1)}$ is a double cover of the group $\se(3)$ of Euclidean motions, see for example \cite{selig:geometricfundamentalsofrobotics}.
\paragraph{Cayley-Klein geometry via homogeneous Clifford algebra models:} In general, a Cayley-Klein space can be described using the homogeneous Clifford algebra approach. All we need is the vector space model of a projective space and a quadratic form given by its signature. Of course, not every Cayley-Klein space can be described in this way, but for our purposes this is convenient. The $3$-dimensional galilei Cayley-Klein space serves as counterexample. It is given by its absolute figure
\[\mathcal{F}:Q_{1,0}^2\supset A^{2}\supset Q_{1,0}^{1}\supset A^{1}\supset Q_{2,0}^{0}.\]
The quadrics and the subspaces can be described by
\[Q_{1,0}^2:\, x_0^2=0,\,\, A^2:\, x_0=0,\,\, Q_{1,0}^1:\,x_1^2=0,\,\, A^1:\,x_1=0,\,\, Q_{2,0}^0:\,x_2^2+x_3^2=0.\]
It is not possible to describe this absolute figure with one quadratic form and therefore we conclude that there exists no homogeneous Clifford algebra representation for the galilei Cayley-Klein space.
\subsection{A homogeneous model for Euclidean geometry}
Now we construct the Clifford algebra for the $3$-dimensional Euclidean space. Therefore, we take $V=\mathds{R}^4$ as model for the $3$-dimensional projective space $\mathds{P}^3(\mathds{R})$. The signature of the Clifford algebra is $(3,0,1)$. Note that in the literature the signature $(0,3,1)$ is often used, because of the connection to dual quaternions, see \cite{selig:geometricfundamentalsofrobotics}. We will do the construction in a more natural way and take the scalar product belonging to Euclidean geometry. We have four basis elements $e_1,\dots, e_4$ with $e_1^2=e_2^2=e_3^2=1$ and $e_4^2=0$. The quadratic form is attached to the grade-1 space, therefore the points are described by grade-3 elements. This approach is called a plane based approach, see \cite{gunn:onthehomogeneousmodel}.  Grade-2 and grade-1 elements correspond to lines and planes, respectively. Note that not every grade-2 element belongs to a line.
The even part of the algebra $\clifford^+_{(3,0,1)}$ is spanned by $e_0,e_{12},e_{13},e_{14},e_{23},e_{24},e_{34},e_{1234}$ and the Spin group is a double cover of $\se(3)$. Note that in this case we have $\clifford^+_{(3,0,1)}\cong\clifford^+_{(0,3,1)}$.
\paragraph{Isomorphism to $\mathds{H}_d$:}
Dual quaternions denoted by $\mathds{H}_d$ are a well-known tool to describe Euclidean displacements in $3$-dimensional space, see \cite{husty:kinematik}. Here we will give a short introduction. Furthermore, we will give an isomorphism between $\clifford^+_{(3,0,1)}$ and $\mathds{H}_d$ as introduced by E. Study, cf. \cite{study:geometriederdynamen}.
\paragraph{Quaternions:} A quaternion, see Example \ref{example1},\ref{example2} and \ref{example3} is a hypercomplex number of the form
\begin{equation*}
\mathfrak{q}= a_0+a_1\mathbf{i}+a_2\mathbf{j}+a_3\mathbf{k},\mbox{ with } a_0,a_1,a_2,a_3\in\mathds{R}.
\end{equation*}
Quaternions were discovered by W.R. Hamilton in 1843. For the quaternion units the relations
\begin{equation}\label{label.2.18}
\mathbf{i}^2=\mathbf{j}^2=\mathbf{k}^2=\mathbf{i}\mathbf{j}\mathbf{k}=-1
\end{equation}
hold. Together with component-wise addition and multiplication defined by Eq.\ \eqref{label.2.18} the set $\mathds{H}$ of quaternions forms a skew field. For more details about quaternions and their relation to kinematics, see \cite{blaschke:kinematikundquaternionen}. 
\paragraph{Dual numbers:} Dual numbers form a commutative ring with unity. A dual number is defined by
\begin{equation*}
d=a+\epsilon b,\mbox{ with } a,b\in\mathds{R} \text{ and } \epsilon^2=0.
\end{equation*}
The set of dual numbers is denoted by $\mathds{D}$. Pure dual numbers have the form $\epsilon b$ and are zero divisors in the ring $\mathds{D}$. This can be seen easily since the dual unit $\epsilon$ commutes with real numbers $(\epsilon a)(\epsilon b)=\epsilon^2 ab=0$. In \cite{marova} dual numbers were efficiently used to extend several well-known theorems from elementary geometry to affine Cayley-Klein planes, see also \cite{weissmann}.
\paragraph{Dual quaternions:} Quaternions with dual number components are called dual quaternions and are denoted by
\begin{align*}
\mathds{H}_d:=&\left\lbrace a_0+a_1\mathbf{i}+a_2\mathbf{j}+a_3\mathbf{k}+\epsilon(c_0+c_1\mathbf{i}+c_2\mathbf{j}+c_3\mathbf{k})\right. \\&\left. \mid a_0,\dots ,a_3,c_0,\dots ,c_3\in\mathds{R} \right\rbrace.\nonumber
\end{align*}
The dual quaternions form an 8-dimensional vector space over the real numbers and the basis elements are $1,\mathbf{i},\mathbf{j},\mathbf{k},\epsilon,\epsilon\mathbf{i},\epsilon\mathbf{j},\epsilon\mathbf{k}$. We give an isomorphism between the even Clifford algebra $\clifford^+_{(3,0,1)}$ and the dual quaternions $\mathds{H}_d$. This isomorphism can also be found in \cite{selig:geometricfundamentalsofrobotics} and is given by its action on the basis:
\begin{align}\label{label.2.20}
e_0 &\mapsto 1,   &e_{23}\mapsto \mathbf{i}, &&e_{31}\mapsto \mathbf{j}, &&e_{12}\mapsto \mathbf{k},\\
 -e_{1234}&\mapsto \epsilon, &e_{14}\mapsto\epsilon\mathbf{i}, &&e_{24}\mapsto\epsilon\mathbf{j}, &&e_{34}\mapsto\epsilon\mathbf{k}.\nonumber
\end{align}
Displacements can be described by unit dual quaternions. A dual quaternion is normed, if it satisfies
\begin{align}\label{label.2.21}
a_0^2+a_1^2+a_2^2+a_3^2=1,\\
a_0c_0+a_1c_1+a_2c_2+a_3c_3=0.\label{label.2.22}
\end{align}
Unit dual quaternions are denoted with $\mathds{U}_d$ and form a group with respect to multiplication. More information can be found in \cite{husty:kinematik}. Furthermore, we can define a group isomorphism between the group of dual unit quaternions and $\mbox{Spin}_{(3,0,1)}$, if we restrict the isomorphism \eqref{label.2.20} to one of these groups in the image or pre-image.
\subsection{Poincar\'{e} Duality}
In the literature duality is often defined through multiplication with the pseudo scalar. When we deal with degenerated Clifford algebras this definition is not adequate. Therefore, we define duality by the so called Poincar\'{e} isomorphism. Gunn \cite{gunn:kinematics} uses this isomorphism to define projective meet and join operations.
\begin{definition}
The isomorphism 
\begin{align*}
J:\clifford_{(p,q,r)}&\to\clifford_{(p,q,r)},\\
e_{i_1\dots i_k}&\mapsto e_{I\backslash\left\lbrace i_1,\dots, i_k\right\rbrace },0\leq i_1<\dots <i_k\leq n,
\end{align*}
where $I$ is the ordered set $\left\lbrace 1,\dots,n\right\rbrace $ and $n=p+q+r$ the dimension of the vector space is called the Poincar\'{e} duality. Grade-$k$ elements are mapped to grade-$(n-k)$ elements.
\end{definition}
\begin{example}
If we apply this mapping to a grade-$2$ element $\mathfrak{p}=x_0e_{12}+x_1e_{23}+x_2e_{13}$ of $\clifford_{(2,0,1)}$ we get a grade-$1$ element $J(\mathfrak{p})=x_0e_{0}+x_1e_{1}+x_2e_{2}$. The scalar product \eqref{label.2.10} results in
\[J(\mathfrak{p})\cdot J(\mathfrak{p})=x_1^2+x_2^2.\]
This expression belongs to the norm of vectors in the Euclidean plane. Note that in this case $J(\mathfrak{p})^2=J(\mathfrak{p})\cdot J(\mathfrak{p})$. It can be interpreted as the squared distance between the origin and the hyperplane $J(\mathfrak{p})$. Therefore, the geometric entity belonging to $J(\mathfrak{p})$ can be interpreted as line, {\it i.e.}, a hyperplane in $\mathds{P}^2(\mathds{R})$.
\end{example}
\section{Kinematic mappings}
In this Section we give a short introduction to the concept of kinematic mappings. We will treat two important examples. Kinematic mappings map displacements to points in a certain space. The main advantage is to work with points instead of displacements.
\subsection{Study's kinematic mapping}
The $6$-dimensional group of Euclidean displacements in $3$-dimensional space can be mapped to a special hyperquadric $S_2^6$ in $\mathds{P}^7(\mathds{R})$ which is usually referred to as Study's quadric. Each displacement is represented by a point on this quadric, but not every point on $S_2^6$ belongs to a displacement. There exists a $3$-dimensional generator space $E^3\subset S_2^6$ whose points do not correspond to displacements. To see this we just have to look at Eq.\ \eqref{label.2.21} for a dual unit quaternion. If we use projective coordinates the first equation is not allowed to vanish. In fact this is the exceptional space $E^3$ given by $a_0=a_1=a_2=a_3=0$ whose points do not belong to any displacement. Note that $N_2^6:a_0^2+a_1^2+a_2^2+a_3^2=0$ is a singular quadric with 3-dimensional real vertex space $E^3$. Eq. \eqref{label.2.22} describes Study's quadric. Therefore, the point set $S_2^6\backslash E^3$ is the image of the Euclidean displacements. Thus, the image space is a sliced quadric $S_2^6\backslash E^3$, a pseudo algebraic variety, so to say. Hence, we have a bijective mapping
\begin{align}
\mathcal{S}:\se(3)&\rightarrow S_2^6\backslash E^3\subseteq\mathds{P}^7(\mathds{R}),\label{studymapping}\\
\se(3)\ni \alpha &\mapsto A=(a_0,\dots ,a_7)^T.\nonumber
\end{align}
The group of collineations in the image space that correspond to $\se(3)$ are the collineations $\pgl(\mathds{P}^7,\left[ S_2^6,N_2^6\right])$ that leave the pencil of quadrics $\left[ S_2^6,N_2^6\right]$ invariant. Note that the image space is not a Cayley-Klein space. On the one hand we could interpret the $7$-dimensional projective space with its absolute figure
\[\mathcal{F}:S_2^6\supset E^3\supset Q^2_{4,0}\]
as Cayley-Klein space when we use the real numbers as underlying field. On the other hand, if we look with complex glasses on the scene, the exceptional space $E^3$ becomes the vertex space of a singular quadric $N_2^6:\sum_{i=0}^3 a_i^2=0$, that belongs to the absolute figure. Furthermore, the quadric $Q_{4,0}^2$ is a quadric in the subspace $E^3$ given by the equation $c_0^2+c_1^2+c_2^2+c_3^2=0$.
\subsection{A mapping for planar displacements}
Here we present another kinematic mapping introduced by Blaschke \cite{blaschke:euklidischekinematik} and Gr\"unwald \cite{gruenwald:einabbildungsprinzip}. The interested reader is referred to \cite[Ch. 11]{bottema:theoreticalkinematics} for further information. Planar Euclidean motions can be factored in a rotation and a translation
\begin{equation*}
\begin{pmatrix}x'\\y'\end{pmatrix}=\left(\begin{array}{cr}\cos\varphi & -\sin\varphi\\\sin\varphi & \cos\varphi\end{array}\right)+\begin{pmatrix}a\\b\end{pmatrix}.
\end{equation*}
The kinematic mapping $\varkappa$ maps planar Euclidean displacements to points of $\mathds{P}^3(\mathds{R})$. Note that the homogeneous coordinates of points are related to Cartesian coordinates $(x,y,z)^T$ via
\[x=\frac{x_1}{x_0},\quad y=\frac{x_2}{x_0},\quad z=\frac{x_3}{x_0},\quad \mbox{if } x_0\neq 0.\]
Again not every point of the image space belongs to a displacement. In analogy to Section 3.1 the line $\ell:\,x_0=x_3=0$
 is the real vertex space of a reducible quadric $N_2^2:\,x_0^2+x_3^2=0$, which consists of the complex conjugate planes $x_0+ix_3=0$ and $x_0-ix_3=0$. Therefore, $\ell$ is the intersection line of these two planes.
To get a bijection we have to remove the line at infinity $\ell$ from the image space. Under these conditions, the map
\begin{align}
\varkappa:\se(2)&\rightarrow \mathds{P^3}(\mathds{R})\backslash\ell\label{blaschkemapping}\\
\se(2)\ni \alpha(\varphi,a,b)&\mapsto\left(2\cos\frac{\varphi}{2},a\sin\frac{\varphi}{2}-b\cos\frac{\varphi}{2},a\cos\frac{\varphi}{2}+b\sin\frac{\varphi}{2},2\sin\frac{\varphi}{2}\right)^T\nonumber
\end{align}
is one-to-one and onto. Hence, the image space is a quasi-elliptic space, cf. \cite{giering:vorlesungenueberhoeheregeometrie} with absolute figure
\[\mathcal{F}:Q^2_{2,0}\supset A^1\supset Q^0_{2,0},\]
\[Q_{2,0}^2:\,x_0^2+x_3^2=0,\,\, A^1:\,x_0=x_3=0,\,\,Q_{2,0}^0:\,x_1^2+x_2^2=0.\]
Planar Euclidean displacements can be described by a subgroup of the group $\mathds{U}_d$. For example: If we want to describe planar Euclidean displacements in the $\left[ x,y\right]$-plane we have to restrict the group $\mathds{U}_d$ to rotations around axes that are parallel to the $z$-axis and to translations in the $\left[ x,y\right]$-plane.
\section{Kinematic mappings via Clifford algebras}
Both examples presented in the previous Section can be treated together in the unifying and more general framework of Clifford algebras. In the following we shall explain how this works and illustrate the construction at hand of the previous examples. The Spin group corresponding to a certain homogeneous Clifford algebra model for a Cayley-Klein space is a double cover of the group of collineations that preserve the absolute figure $\mathcal{F}$. Therefore, we just have to look at the Spin group as a subset of a projective space of the dimension $2^n-1$. 
\subsection{Study's mapping via Clifford algebra}
Again we start with the homogeneous model for the $3$-dimensional Euclidean space $\clifford_{(3,0,1)}$. The Spin group is located in the even part $\clifford^+_{(3,0,1)}$ of the algebra and every Spin group element satisfies the condition $N(\mathfrak{g})=\mathfrak{gg}^\ast=1$. Thus, we take an arbitrary element of the even part of the algebra
\[\mathfrak{g}=a_{0}+a_{3}e_{12}+a_{2}e_{13}+c_{1}e_{14}+a_{1}e_{23}+c_{2}e_{24}+c_{3}e_{34}+c_{0}e_{1234}\]
and calculate the product with its conjugate element
\begin{equation}\label{label.4.1}
\mathfrak{gg}^\ast=(a_{0}^2+a_{1}^2+a_{2}^2+a_{3}^2)+2(c_{0}a_{0}-a_{1}c_{1}+c_{2}a_{2}-c_{3}a_{3})e_{1234}=1.
\end{equation}
Eq.\ \eqref{label.4.1} shows that $\mathfrak{g}\in\mbox{Spin}_{(3,0,1)}$, if the quadratic equation in the pseudo scalar part is equal to zero. With the isomorphism \eqref{label.2.20} we can see that this condition is exactly the equation of Study's quadric. The values $a_i,c_i,\. i=0,\dots,3$ are the coordinates of a Euclidean displacement, for
\begin{equation}\label{Study}
c_{0}a_{0}-a_{1}c_{1}+c_{2}a_{2}-c_{3}a_{3}=0.
\end{equation}
 Furthermore, the scalar part of Eq.\ \eqref{label.4.1} has to be equal to $1$. Now we assume that the $8$ components $a_0,\dots ,a_3,c_0,\dots ,c_3$ are coordinates of a $7$-dimensional projective space $\mathds{P}^7(\mathds{R})$. The condition in the scalar part can now be relaxed to
\[a_{0}^2+a_{1}^2+a_{2}^2+a_{3}^2\neq 0.\]
This equation can be identified as the condition that a point is not allowed to lie in the exceptional space $E^3$. Hence, we have found a mapping from $\mbox{Spin}_{(3,0,1)}\rightarrow S_2^6\backslash E^3$, compare to Eq. \eqref{studymapping}.
Note that $\mathfrak{g}$ and $-\mathfrak{g}$ represent the same displacement.
\paragraph{Matrices of displacements:}
In this paragraph we show how to describe the collineations of $\mathds{P}^3(\mathds{R})$ that preserve the absolute figure \eqref{3deuclidean} of the Euclidean space seen as Cayley-Klein space. To do so, we look at the action of the Spin group on a point $P=(x_0,x_1,x_2,x_3)^T$ that is represented by a grade-3 element \[\mathfrak{p}=x_0e_{123}+x_1e_{234}+x_2e_{134}+x_3e_{124}.\] An arbitrary element of the Spin group is given by 
\[\mathfrak{g}=a_{0}+a_{1}e_{23}+a_{2}e_{13}+a_{3}e_{12}+c_{0}e_{1234}+c_{1}e_{14}+c_{2}e_{24}+c_{3}e_{34},\quad \mathfrak{gg}^\ast=1.\]
In the projective representation the condition $\mathfrak{gg}^\ast=1$ changes to $\sum_{i=0}^3{a_i^2}\neq 0$ and further we have $a_0c_0-a_1c_1+a_2c_2-a_3c_3=0$. The effect of a Spin group element can be written in terms of the sandwich operator
\begin{align*}
\mathfrak{gpg}^\ast&=( a_0^2+a_1^2+a_2^2+a_3^2)x_0e_{123}\\
&+\big(2(-a_0c_1-a_1c_0-a_2c_3-a_3c_2)x_0+(a_0^2+a_1^2-a_2^2-a_3^2)x_1\\
&+2(a_2a_1-a_0a_3)x_2+2(a_0a_2+a_3a_1)x_3\big)e_{234}\\
&+\big(2(a_0c_2+a_1c_3-a_2c_0-a_3c_1)x_0+2(a_0a_3+a_2a_1)x_1\\
&+(a_0^2-a_1^2+a_2^2-a_3^2)x_2+2(a_3a_2-a_0a_1)x_3\big)e_{134}\\
&+\big((2a_2c_1+2a_1c_2-2a_0c_3-2a_3c_0)x_0+2(a_3a_1-a_0a_2)x_1\\
&+2(a_3a_2+a_0a_1)x_2+(a_0^2-a_1^2-a_2^2+a_3^2)x_3\big)e_{124}.
\end{align*}
Here we see that the action of $\mathfrak{g}$ on $\mathfrak{p}$ is linear. Hence, we can write this action as a product of a matrix with a vector
\[\begin{pmatrix}x_0\\x_1\\x_2\\x_3\end{pmatrix}'=\frac{1}{\Delta}\begin{pmatrix}
 \Delta& 0 & 0 & 0\\
  l & a_{0}^2+a_{1}^2-a_{2}^2-a_{3}^2 & 2(a_{2}a_{1}-a_{0}a_{3}) & 2(a_{0}a_{2}+a_{3}a_{1})\\
  m & 2(a_{0}a_{3}+a_{2}a_{1}) & a_{0}^2-a_{1}^2+a_{2}^2-a_{3}^2 & 2(a_{3}a_{2}-a_{0}a_{1})\\
  n & 2(a_{3}a_{1}-a_{0}a_{2}) & 2(a_{3}a_{2}+a_{0}a_{1}) & a_{0}^2-a_{1}^2-a_{2}^2+a_{3}^2
\end{pmatrix}\begin{pmatrix}x_0\\x_1\\x_2\\x_3\end{pmatrix}\]
with
\begin{align*}
\Delta&=a_{0}^2+a_{1}^2+a_{2}^2+a_{3}^2, &l&=2(-a_{0}c_{1}-a_{3}c_{2}-a_{2}c_{3}-a_{1}c_{0}),\\
m&=2(a_{0}c_{2}+a_{1}c_{3}-a_{2}c_{0}-a_{3}c_{1}), &n&=2(a_{2}c_{1}+a_{1}c_{2}-a_{0}c_{3}-a_{3}c_{0}).
\end{align*}
The factor $\frac{1}{\Delta}$ guarantees that the right lower $3\times 3$ submatrix is in $\so(3)$.
\paragraph{Collineations in the image space:}
Transformations in the underlying geometry induce collineations in the kinematic image space. Therefore, we are interested in the matrix representation of Spin group elements, see Eq. \eqref{matrixrepresentation}. Thus, we take an element from $\mbox{Spin}_{(3,0,1)}$ \[\mathfrak{g}=a_{0}+a_{1}e_{23}+a_{2}e_{13}+a_{3}e_{12}+c_{0}e_{1234}+c_{1}e_{14}+c_{2}e_{24}+c_{3}e_{34}.\]
When using homogeneous coordinates the norming condition becomes\linebreak
$\sum_{i=0}^{3}{a_i^2}\neq 0$. Furthermore, the parameters $(a_0,\dots ,a_3,c_0,\dots ,c_3)$ have to satisfy the condition \eqref{Study}. We get
\begin{align*}
\mathfrak{g}e_0&=  a_{0} e_0 + a_{3} e_{12} + a_{2} e_{13} + c_{1} e_{14} + a_{1} e_{23} + c_{2} e_{24}
+ c_{3} e_{34} + c_{0} e_{1234},\\
\mathfrak{g}e_{23}&=  -a_{1} e_0 - a_{2} e_{12} + a_{3} e_{13} - c_{0} e_{14} + a_{0} e_{23} - c_{3} e_{24}
+ c_{2} e_{34} + c_{1} e_{1234},\\
\mathfrak{g}e_{13}&=  -a_{2} e_0 + a_{1} e_{12} + a_{0} e_{13} - c_{3} e_{14} - a_{3} e_{23} + c_{0} e_{24}
+ c_{1} e_{34} - c_{2} e_{1234},\\
\mathfrak{g}e_{12}&=  -a_{3} e_0 + a_{0} e_{12} - a_{1} e_{13} - c_{2} e_{14} + a_{2} e_{23} + c_{1} e_{24}
- c_{0} e_{34} + c_{3} e_{1234},\\
\mathfrak{g}e_{1234}&= -a_{1} e_{14} + a_{2} e_{24} - a_{3} e_{34} + a_{0} e_{1234},\\
\mathfrak{g}e_{14}&=    a_{0} e_{14} - a_{3} e_{24} - a_{2} e_{34} + a_{1} e_{1234},\\
\mathfrak{g}e_{24}&=    a_{3} e_{14} + a_{0} e_{24} - a_{1} e_{34} - a_{2} e_{1234},\\
\mathfrak{g}e_{34}&=  a_{2} e_{14} + a_{1} e_{24} + a_{0} e_{34} + a_{3} e_{1234}.
\end{align*}
If we write these equations in matrix form we obtain a representation of collineations of $\mathds{P}^7(\mathds{R})$ that preserve the pencil of quadrics spanned by $\left[S_2^6,N_2^6\right]$:
\[\left[ G^+\right]= \begin{pmatrix}
 a_0&  -a_1& -a_2& -a_3& 0   &    0 &0    &0   \\
 a_1& a_0  & -a_3&  a_2& 0   &    0 &   0 &    0 \\
 a_2& a_3  & a_0 &-a_1 & 0   &    0 &   0 &    0 \\
 a_3&  -a_2& a_1 & a_0 & 0   &    0 &   0 &   0 \\
 c_0&  c_1 &-c_2 & c_3 & a_0 & a_1  & -a_2& a_3 \\
 c_1& -c_0 & -c_3&-c_2 & -a_1& a_0  &  a_3& a_2\\
 c_2&  -c_3& c_0 &  c_1& a_2 & -a_3 & a_0 & a_1\\
 c_3 & c_2 &c_1  & -c_0& -a_3&  -a_2& -a_1&  a_0
\end{pmatrix}.\]
In order to write $\left[ G^+ \right] $ in the usual form, see \cite{pfurner:analysis}, the isomorphism \eqref{label.2.20} has to be used.
\subsection{Blaschke's and Gr\"unwald's mapping via Clifford algebra}
For planar Euclidean geometry the homogeneous Clifford algebra model $\clifford_{(2,0,1)}$ is adequate. We find the Spin group in the even part $\clifford^+_{(2,0,1)}$. A general element of the even part is given by
\[\mathfrak{g}=a_0e_0+a_{1}e_{12}+c_0e_{23}+c_1e_{13}.\]
The condition that this element is a Spin group element reads now
\[\mathfrak{gg}^\ast=a_0^2+a_1^2=1.\]
In this case we have just one equation in the scalar part. If we change to projective coordinates, the condition gets $a_0^2+a_1^2\neq 0$. Again we have the equation of a pair of complex conjugate planes intersecting in a real line. This results in the same image space as in the previous Section, see Eq. \eqref{blaschkemapping}.
\paragraph{Matrices of planar Euclidean displacements:}
With the Spin group we can find the matrix representation corresponding to the group $\se(2)$ as collineation group fixing the absolute figure \eqref{label.2.17}. All we have to do is to apply the sandwich operator to a point. As we choose a plane based construction a point $P=(x_0,x_1,x_2)$ can be described in the algebra $\clifford_{(2,0,1)}$ via $\mathfrak{p}=x_0e_{12}+x_1e_{23}+x_2e_{13}$. A general element of the Spin group is given by $\mathfrak{g}=a_0e_0+a_{1}e_{12}+c_0e_{23}+c_1e_{13}$. Applying the sandwich operator to $\mathfrak{p}$ results in
\begin{align*}
\mathfrak{gpg}^\ast&=\big(a_0^2+a_1^2\big)x_0e_{12}+\big(2(a_1c_0+a_0c_1)x_0+(a_0^2-a_1^2)x_1-2a_0a_1x_2\big)e_{23}\\
&+ \big(2(a_1c_1-a_0c_0)x_0+2a_0a_1x_1+(a_0^2+a_1^2)x_2\big)e_{13}.
\nonumber\end{align*}
If we rewrite the effect of this transformation as linear transformation of the projective plane, we get
\begin{equation}\label{label.4.3}
\begin{pmatrix}x_0\\x_1\\x_2\end{pmatrix}'=\frac{1}{a_0^2+a_{1}^2}
\begin{pmatrix} 
a_{0}^2+a_{1}^2 & 0 & 0\\
2(a_{1}c_{0}+a_{0}c_{1}) & a_{0}^2-a_{1}^2 & -2a_{0}a_{1}\\
2(a_{1}c_{1}-c_{0}a_{0}) & 2a_{0}a_{1} & a_{0}^2-a_{1}^2
\end{pmatrix}\cdot
\begin{pmatrix} x_0\\x_1\\x_2\end{pmatrix}.
\end{equation}
In the first line of Eq. \eqref{label.4.3} we see that the line at infinity is fixed under these collineations. It is clear that all planar Euclidean displacements are obtained in this way, because the Spin group is a double cover of $\se(2)$. The factor is added 
artificially to guarantee that we just look at matrices belonging to the Spin group. These two examples show that the Spin group can be used to get the matrix group for a Cayley-Klein geometry, if it is representable as homogeneous Clifford algebra model.
\paragraph{Collineations in the image space:}
The Spin group elements induce in the image space a group of transformations that fix the pair of complex conjugate planes, their intersection line and the two complex conjugate points on this line. To get this group we have to look at the matrix representation of the Spin group. Let $\mathfrak{g}=a_0e_0+a_1e_{12}+c_0e_{23}+c_1e_{13}\in\mbox{Spin}_{(2,0,1)}$. Calculating the matrix group as described in Section \ref{matrixsection} results in
\[\pgl(\mathds{P}^2(\mathds{R}),\mathcal{F})=\left\lbrace \begin{pmatrix} a_0 & -a_1 &0 &0\\
a_1& a_0&0&0\\c_1&-c_0&a_0&a_1\\c_0&c_1&-a_1&a_0 
\end{pmatrix}|a_0,a_1,c_0,c_1\in\mathds{R}\right\rbrace.\]
\section{Kinematic mappings for other Cayley-Klein spaces}
This procedure can be applied to homogeneous Clifford algebra models of any signature. We will present one more example here. For the remaining signatures, see Table \ref{table1} and Table \ref{table2}. Note that it makes no sense to look at the completely degenerate Clifford algebra $\clifford_{(0,0,3)}$ because the scalar product of any two vectors is zero.
\subsection{Two-dimensional Cayley-Klein spaces}
\paragraph{Elliptic plane:}
The homogeneous model for the elliptic plane is given by\linebreak $\clifford_{(3,0,0)}$. Let $\mathfrak{g}=a_0e_0+a_1e_{12}+c_0e_{23}+c_1e_{13}$ be an element of the even part $\clifford^+_{(3,0,0)}$. The product of this element with its conjugate element yields
\[\mathfrak{gg}^\ast=a_0^2+a_1^2+c_0^2+c_1^2=1.\]
In homogeneous coordinates this condition gets $a_0^2+a_1^2+c_0^2+c_1^2\neq 0$. Therefore, we can interpret the kinematic image of the elliptic congruences as non-degenerated Cayley-Klein space with absolute figure $\mathcal{F}:Q^2_{4,0}$. Furthermore, this means, if we work over the real numbers every point has a pre-image. In Table \ref{table1} the absolute figures $\mathcal{F}$ and the corresponding Clifford algebras for possible $2$-dimensional Cayley-Klein spaces are given.
{\linespread{1.2}
\begin{table}[h]
\begin{center}
\begin{tabular}{|c|c|c|}
\hline 
\textbf{CK space} & \textbf{absolute Figure} $\mathcal{F}$ & \textbf{Clifford algebra} \\ 
\hline 
elliptic & $Q_{3,0}^1$ & $\clifford_{(3,0,0)},\,\clifford_{(0,3,0)}$ \\ 
\hline 
hyperbolic & $Q_{3,1}^1$ & $\clifford_{(2,1,0)},\,\clifford_{(1,2,0)}$ \\ 
\hline 
Euclidean & $Q_{1,0}^1\supset A^1\supset Q_{2,0}^0$ & $\clifford_{(2,0,1)},\,\clifford_{(0,2,1)}$ \\ 
\hline 
pseudo-Euclidean & $Q_{1,0}^1\supset A^1\supset Q_{2,1}^0$ & $\clifford_{(1,1,1)}$ \\ 
\hline 
quasi-elliptic & $Q_{2,0}^1\supset A^0\supset Q_{1,0}^{-1}$ & $\clifford_{(2,0,1)},\,\clifford_{(0,2,1)}$ \\ 
\hline 
quasi-hyperbolic & $Q_{2,1}^1\supset A^0\supset Q_{1,0}^{-1}$ & $\clifford_{(1,1,1)}$ \\ 
\hline 
totally isotr. space & $Q_{1,0}^1\supset A^1\supset Q_{1,0}^0\supset A^0\supset Q_{1,0}^{-1}$ & $\clifford_{(1,0,2)},\,\clifford_{(0,1,2)}$ \\ 
\hline 
\end{tabular} 
\caption{Planar Cayley-Klein geometries and belonging Clifford algebras}
\label{table1}
\end{center}
\end{table}}
\FloatBarrier
\noindent
Note that the points of the Cayley-Klein geometry are always written in the grade-2 subspace of the Clifford algebra except the quasi-elliptic and the quasi-hyperbolic case. In these two cases we use the Poincaré isomorphism to describe these Cayley-Klein geometries as dual partners of the Euclidean respectively, the pseudo-Euclidean Cayley-Klein geometry. Therefore, points of the quasi-elliptic (quasi-hyperbolic) Cayley-Klein geometry are represented as grade-1 elements in the algebra corresponding to the Euclidean (pseudo-Euclidean) Cayley-Klein space. Table \ref{table2} shows the kinematic image spaces of the seven planar Cayley-Klein spaces. Note that every image space is again a Cayley-Klein space.
{\linespread{1.2}
\begin{table}[h]
\begin{center}
\begin{tabular}{|c|c|c|}
\hline 
\textbf{pre-image} & \textbf{CK image space} & \textbf{absolute Figure} $\mathcal{F}$ \\ 
\hline 
elliptic & elliptic & $Q_{4,0}^2$ \\ 
\hline 
hyperbolic & hyperbolic idx. $1$ & $Q_{4,2}^2$ \\ 
\hline 
Euclidean & quasi-elliptic & $Q_{2,0}^2\supset A^1\supset Q_{2,0}^0$ \\ 
\hline 
pseudo-Euclidean & quasi-hyperbolic idx. $0$ & $Q_{2,1}^2\supset A^1\supset Q_{2,1}^0$ \\ 
\hline 
quasi-elliptic & quasi-elliptic & $Q_{2,0}^2\supset A^1\supset Q_{2,0}^0$ \\ 
\hline 
quasi-hyperbolic & quasi-hyperbolic idx. $0$ & $Q_{2,1}^2\supset A^1\supset Q_{2,1}^0$ \\ 
\hline 
totally isotr. space& totally isotr. space & $Q_{1,0}^2\supset A^2\supset Q_{2,0}^1\supset A^0\supset Q_{1,0}^{-1}$ \\ 
\hline 
\end{tabular}
\caption{Kinematic image spaces of $2$-dimensional Cayley-Klein geometries presented as $3$-dimensional Cayley-Klein spaces with their absolute figure $\mathcal{F}$}
\label{table2} 
\end{center}
\end{table}}

\subsection{Three-dimensional Cayley-Klein spaces}

The kinematic image space for the $3$-dimensional Euclidean space was constructed in Section 4.1. Here we want to repeat this construction for the elliptic $3$-dimensional Cayley-Klein space. In \cite[Ch. 11]{selig:geometricfundamentalsofrobotics} the author showed that elements of the group $\so(4)$ can be identified with points of Study's quadric. In this case no exceptions have to be made. Therefore, the kinematic mapping is one-to-one and onto. Let us look what happens, if we construct the kinematic mapping belonging to the $3$-dimensional elliptic Cayley-Klein geometry which is modelled through $\clifford_{(4,0,0)}$. This Clifford algebra is $16$-dimensional and the even part $\clifford^+_{(4,0,0)}$ is $8$-dimensional. An arbitrary Spin group element is given by
\[\mathfrak{g}=a_{0}+a_{1}e_{23}+a_{2}e_{13}+a_{3}e_{12}+c_{0}e_{1234}+c_{1}e_{14}+c_{2}e_{24}+c_{3}e_{34}.\]
The condition that it is an element of the Spin group reads
\[\mathfrak{gg}^\ast=(a_0^2+a_1^2+a_2^2+a_3^2+c_0^2+c_1^2+c_2^2+c_3^2)e_0+(a_0c_0-a_1c_1+a_2c_2-a_3c_3)e_{1234}=1.\]
If we change to $7$-dimensional projective space the first condition is that the scalar part should not vanish
\[a_0^2+a_1^2+a_2^2+a_3^2+c_0^2+c_1^2+c_2^2+c_3^2\neq 0.\]
As in the case for Euclidean displacements this equation defines an exceptional set, {\it i.e.}, points that do not correspond to a pre-image. Here we can write this exceptional set as hyper quadric
\[Q_2^6:a_0^2+a_1^2+a_2^2+a_3^2+c_0^2+c_1^2+c_2^2+c_3^2=0.\]
The term in the pseudo scalar part is again the equation of Study's quadric
\[S_2^6:a_0c_0-a_1c_1+a_2c_2-a_3c_3=0.\]
Every point on Study's quadric stands for an element of $\so(4)$ since $\mbox{Spin}_{(4,0,0)}$ is a double cover of $\so(4)$. In this case it is not necessary to slice the quadric, since $Q_2^6$ has no real point. Furthermore, this means the pencil of quadrics spanned by $\left[S_2^6,Q_2^6 \right] $ is the absolute figure of the kinematic image space. Surprisingly, every kinematic mapping of $3$-dimensional Cayley-Klein spaces maps the Spin group elements to points on Study's quadric. In the following we will list the possible Cayley-Klein spaces and the absolute quadric pencil in the kinematic image space.
In Table \ref{table3} the other possible homogeneous Clifford algebra models are presented. Note that every Cayley-Klein space is self-dual except the Euclidean and the pseudo-Euclidean. Their dual Cayley-Klein spaces can be represented in the dual homogeneous Clifford algebra model obtained through the Poincaré duality.
\FloatBarrier
{\linespread{1.2}
\begin{table}[hbt]
\begin{center}
\begin{tabular}{|c|c|c|c|}
\hline 
\textbf{CK-space} & $\mathbf{\clifford_{(p,q,r)}}$ & $\mathbf{Q_2^6}$  \\ 
\hline 
elliptic & $\clifford_{(4,0,0)},\,\clifford_{(0,4,0)}$ & $\sum_{i=0}^3 a_i^2+\sum_{i=0}^3 c_i^2$  \\ 
\hline 
hyperbolic idx. 0 & $\clifford_{(3,1,0)},\,\clifford_{(1,3,0)}$ & $\sum_{i=0}^3 a_i^2-\sum_{i=0}^3 c_i^2$  \\ 
\hline 
hyperbolic idx. 1 & $\clifford_{(2,2,0)}$ & $a_0^2-a_1^2-a_2^2+a_3^2+c_0^2-c_1^2-c_2^2+c_3^2$ \\ 
\hline 
Eucledian & $\clifford_{(3,0,1)},\,\clifford_{(0,3,1)}$ & $\sum_{i=0}^3 a_i^2$  \\ 
\hline 
pseudo-Euclidean & $\clifford_{(2,1,1)},\,\clifford_{(1,2,1)}$ & $a_0^2-a_1^2-a_2^2+a_3^2$\\ 
\hline 
quasi-elliptic & $\clifford_{(2,0,2)},\,\clifford_{(0,2,2)}$ & $a_0^2+a_3^2$ \\ 
\hline 
quasi-hyperbolic idx. 0 & $\clifford_{(1,1,2)}$ & $a_0^2-a_3^2$ \\ 
\hline 
double isotr. flagspace & $\clifford_{(1,0,3)},\,\clifford_{(0,1,3)}$ & $a_0^2$  \\ 
\hline 
\end{tabular} 
\caption{3-dimensional Cayley-Klein spaces with possible Clifford algebra representation and exceptional quadric $Q_2^6$}
\label{table3}
\end{center}
\end{table}}
\FloatBarrier
\subsection{Higher dimensional kinematic mappings}
The procedure presented above can be generalized to arbitrary dimensions. Here we present just one example, the $4$-dimensional Euclidean space. Written as homogeneous Clifford algebra model we get $\clifford_{(4,0,1)}$. The dimension of the algebra is $32$ and therefore the dimension of the even part is $16$. Now we take an arbitrary element of the even part and formulate the conditions that it is element of $\mbox{Spin}_{(4,0,1)}$. Note that in this case two conditions are necessary, {\it i.e.}, $\mathfrak{gg}^\ast=\mathfrak{g}^\ast \mathfrak{g}=1$. The element has the form
\begin{align*}
\mathfrak{g} &=a_0e_0+a_1e_{23}+a_2e_{13}+a_3e_{12}+a_4e_{15}+a_5e_{45}+a_6e_{25}+a_7e_{35}\\
&+c_1e_{14}+c_2e_{24}+c_3e_{34}+c_4e_{1235}+c_5e_{1345}+c_6e_{1245}+c_7e_{2345}+c_0e_{1234}.
\end{align*}
The product $\mathfrak{gg}^\ast$ calculates to
\begin{align}\label{label5.1}
\mathfrak{gg}^\ast &=(c_{2}^2+c_{3}^2+a_{2}^2+a_{3}^2+a_{0}^2+c_{1}^2+a_{1}^2+c_{0}^2)e_0\\
&+2(c_{0}a_{0}-a_{1}c_{1}-c_{3}a_{3}+c_{2}a_{2})e_{1234}\nonumber\\
&+2(c_{7}c_{1}+c_{4}a_{0}-c_{5}c_{2}-c_{0}a_{5}+a_{6}a_{2}-a_{1}a_{4}+c_{6}c_{3}-a_{7}a_{3})e_{1235}\nonumber\\
&+2(c_{5}a_{1}+a_{6}c_{1}-a_{5}a_{3}-c_{2}a_{4}-c_{7}a_{2}+c_{6}a_{0}+c_{0}a_{7}-c_{4}c_{3})e_{1245}\nonumber\\
&+2(c_{4}c_{2}-c_{6}a_{1}-a_{5}a_{2}-c_{3}a_{4}-c_{0}a_{6}+c_{5}a_{0}+c_{7}a_{3}+a_{7}c_{1})e_{1345}\nonumber\\
&+2(c_{6}a_{2}+c_{0}a_{4}-c_{4}c_{1}-c_{5}a_{3}+a_{7}c_{2}-c_{3}a_{6}+c_{7}a_{0}-a_{5}a_{1})e_{2345}.\nonumber
\end{align}
Furthermore we have to calculate $\mathfrak{g}^\ast \mathfrak{g}$
\begin{align}\label{label5.2}
\mathfrak{g^\ast g}&=(a_{1}^2+c_{2}^2+c_{3}^2+c_{0}^2+c_{1}^2+a_{3}^2+a_{2}^2+a_{0}^2)e_0\\
&+2(c_{2}a_{2}-c_{3}a_{3}+c_{0}a_{0}-a_{1}c_{1})e_{1234}\nonumber\\
&+2(c_{4}a_{0}-a_{7}a_{3}-c_{7}c_{1}+c_{0}a_{5}-a_{1}a_{4}+a_{6}a_{2}-c_{6}c_{3}+c_{5}c_{2})e_{1235}\nonumber\\
&+2(c_{7}a_{2}+a_{6}c_{1}-c_{5}a_{1}-a_{5}a_{3}+c_{6}a_{0}-c_{2}a_{4}+c_{4}c_{3}-c_{0}a_{7})e_{1245}\nonumber\\
&+2(c_{0}a_{6}-c_{3}a_{4}-c_{4}c_{2}+c_{6}a_{1}-c_{7}a_{3}+a_{7}c_{1}+c_{5}a_{0}-a_{5}a_{2})e_{1345}\nonumber\\
&+2(a_{7}c_{2}+c_{4}c_{1}-a_{5}a_{1}+c_{5}a_{3}-c_{6}a_{2}-c_{3}a_{6}-c_{0}a_{4}+c_{7}a_{0})e_{2345}.\nonumber
\end{align}
Eq. \eqref{label5.1} and \eqref{label5.2} are the conditions for an even grade element to lay in the Spin group. All coefficients of grade-4 elements have to vanish and the scalar part must be different from zero. Some quadric equations occur in both expressions. All in all the kinematic image space is a projective variety in $\mathds{P}^{15}(\mathds{R})$ sliced along a quadric. It can be written as the intersection of nine quadrics minus one quadric and therefore, as pseudo algebraic variety:
\[\mathcal{V}:\bigcap\limits_{i=1}^9{Q_i}\backslash N_1 \subseteq\mathds{P}^{15}(\mathds{R}),\]
with
\begin{align*}
N_1 &:c_{2}^2+c_{3}^2+a_{2}^2+a_{3}^2+a_{0}^2+c_{1}^2+a_{1}^2+c_{0}^2=0,\\
Q_1&:c_{0}a_{0}-a_{1}c_{1}-c_{3}a_{3}+c_{2}a_{2}=0,\\
Q_2&:c_{7}c_{1}+c_{4}a_{0}-c_{5}c_{2}-c_{0}a_{5}+a_{6}a_{2}-a_{1}a_{4}+c_{6}c_{3}-a_{7}a_{3}=0,\\
Q_3&:c_{5}a_{1}+a_{6}c_{1}-a_{5}a_{3}-c_{2}a_{4}-c_{7}a_{2}+c_{6}a_{0}+c_{0}a_{7}-c_{4}c_{3}=0,\\
Q_4&:c_{4}c_{2}-c_{6}a_{1}-a_{5}a_{2}-c_{3}a_{4}-c_{0}a_{6}+c_{5}a_{0}+c_{7}a_{3}+a_{7}c_{1}=0,\\
Q_5&:c_{6}a_{2}+c_{0}a_{4}-c_{4}c_{1}-c_{5}a_{3}+a_{7}c_{2}-c_{3}a_{6}+c_{7}a_{0}-a_{5}a_{1}=0,\\
Q_6&:c_{4}a_{0}-a_{7}a_{3}-c_{7}c_{1}+c_{0}a_{5}-a_{1}a_{4}+a_{6}a_{2}-c_{6}c_{3}+c_{5}c_{2}=0,\\
Q_7&:c_{7}a_{2}+a_{6}c_{1}-c_{5}a_{1}-a_{5}a_{3}+c_{6}a_{0}-c_{2}a_{4}+c_{4}c_{3}-c_{0}a_{7}=0,\\
Q_8&:c_{0}a_{6}-c_{3}a_{4}-c_{4}c_{2}+c_{6}a_{1}-c_{7}a_{3}+a_{7}c_{1}+c_{5}a_{0}-a_{5}a_{2}=0,\\
Q_9&:a_{7}c_{2}+c_{4}c_{1}-a_{5}a_{1}+c_{5}a_{3}-c_{6}a_{2}-c_{3}a_{6}-c_{0}a_{4}+c_{7}a_{0}=0.
\end{align*}
\begin{remark}
With this method it is also possible to construct the mapping $\so(3)\to S^3$. In this case we do not need the homogeneous Clifford algebra model. Thus, we take the Clifford algebra $\clifford_{(3,0,0)}$ and do the same construction for a Spin group element.
\end{remark}
\section{Projective varieties via kinematic algebra elements}
Here, we present another method to construct projective varieties belonging to the Spin group via Clifford algebra. Therefore, a definition is needed.
\begin{definition}\label{DEF3}
An element $\mathfrak{g}$ of a Clifford algebra $\clifford_{(p,q,r)}$ is called \emph{kinematic}, if it fulfils the following equation:
\[\mathfrak{g}^2=tr(\mathfrak{g})\mathfrak{g}-N(\mathfrak{g}),\]
where $tr(\mathfrak{g})=\mathfrak{g+g^\ast}$ is the trace of the element.
\end{definition}
\noindent
Definition \ref{DEF3} generalizes the definition of kinematic algebras over fields, see \cite{karzel:kinematischeAlgebren}.
To obtain projective varieties by using this definition, we first show that every Spin group element is kinematic.
\begin{lemma}
Spin group elements $\mathfrak{g}\in\clifford_{(p,q,r)}$ are kinematic, {\it i.e.}, they fulfil the equation
\[\mathfrak{g}^2=tr(\mathfrak{g})\mathfrak{g}-N(\mathfrak{g}).\]
\end{lemma}
\begin{proof}
The proof is done by direct calculation.
\[\mathfrak{g}^2=tr(\mathfrak{g})\mathfrak{g}-N(\mathfrak{g})=\mathfrak{(g+g^\ast)g-gg^\ast=gg+g^\ast g-gg^\ast=gg}.\]
Here we have used, that $\mathfrak{gg^\ast=g^\ast g}=1$ for every Spin group element.
\end{proof}
\noindent
Now we ask for general conditions that have to be satisfied by kinematic elements. We aim at a projective variety that belongs to the Spin group. Thus, we take an arbitrary element $\mathfrak{g}\in\clifford^+_{(p,q,r)}$ and calculate the conditions from
\[\mathfrak{gg}^\ast-tr(\mathfrak{g})\mathfrak{g}-N(\mathfrak{g})=0.\]
This results in quadratic equations in several generators that all have to vanish. Now we use these equations and check the condition $\mathfrak{gg}^\ast\neq 0$ to get more conditions that will guarantee that the element is in the Spin group. To understand this method we present the $4$-dimensional Euclidean displacements, {\em i.e.}, $\mbox{Spin}_{(4,0,1)}$.
\begin{example}
A general element $\mathfrak{g}\in\clifford^+_{(4,0,1)}$ has the form
\begin{align*}
\mathfrak{g} &=a_0e_0+a_1e_{23}+a_2e_{13}+a_3e_{12}+a_4e_{15}+a_5e_{45}+a_6e_{25}+a_7e_{35}\\
&+c_1e_{14}+c_2e_{24}+c_3e_{24}+c_4e_{1235}+c_5e_{1345}+c_6e_{1245}+c_7e_{2345}+c_0e_{1234}.
\intertext{For kinematic elements we have}
0&=\mathfrak{gg}^\ast-tr(\mathfrak{g})-N(\mathfrak{g}).\\
\end{align*}
In expanded form this condition reads
\begin{align*}
0&=4(c_{6}a_{2}+c_{0}a_{4}-c_{4}c_{1}-c_{5}a_{3})e_{2345}-4(c_{5}c_{2}-c_{7}c_{1}-c_{6}c_{3}+c_{0}a_{5})e_{1235}\\
&+4(c_{0}a_{7}-c_{7}a_{2}+c_{5}a_{1}-c_{4}c_{3})e_{1245}-4(c_{0}a_{6}+c_{6}a_{1}-c_{4}c_{2}-c_{7}a_{3})e_{1345}.
\intertext{Now we use these four quadratic equations and formulate the Spin group condition $\mathfrak{gg}^\ast=1$,}
\mathfrak{gg}^\ast&=(a_{1}^2+c_{2}^2+c_{3}^2+c_{0}^2+c_{1}^2+a_{0}^2+a_{3}^2+a_{2}^2)e_0\\
&+2(c_{0}a_{0}-a_{1}c_{1}-c_{3}a_{3}+c_{2}a_{2})e_{1234}+2(c_{7}a_{0}-c_{3}a_{6}+a_{7}c_{2}-a_{5}a_{1})e_{2345}\\
&+2(c_{4}a_{0}-a_{7}a_{3}+a_{6}a_{2}-a_{1}a_{4})e_{1235}+2(c_{6}a_{0}-a_{5}a_{3}+a_{6}c_{1}-c_{2}a_{4})e_{1245}\\
&+2(a_{7}c_{1}-a_{5}a_{2}-c_{3}a_{4}+c_{5}a_{0})e_{1345}
\end{align*}
and get five more quadratic equations that have to be fulfilled. All in all we have the corresponding projective variety described by nine quadratic equations and one exceptional quadric with equation
\begin{align*}
N_1&:a_{0}^2+a_{1}^2+a_{2}^2+a_{3}^2+c_{0}^2+c_{1}^2+c_{2}^2+c_{3}^2=0,\\
R_1&:c_{6}a_{2}+c_{0}a_{4}-c_{4}c_{1}-c_{5}a_{3}=0,\\
R_2&:-c_{7}c_{1}-c_{6}c_{3}+c_{5}c_{2}+c_{0}a_{5}=0,\\
R_3&:-c_{7}a_{2}+c_{0}a_{7}+c_{5}a_{1}-c_{4}c_{3}=0,\\
R_4&:c_{0}a_{6}+c_{6}a_{1}-c_{4}c_{2}-c_{7}a_{3}=0,\\
R_5&:-a_{1}c_{1}-c_{3}a_{3}+c_{0}a_{0}+c_{2}a_{2}=0,\\
R_6&:-c_{3}a_{6}+a_{7}c_{2}+c_{7}a_{0}-a_{5}a_{1}=0,\\
R_7&:c_{4}a_{0}-a_{7}a_{3}+a_{6}a_{2}-a_{1}a_{4}=0,\\
R_8&:c_{6}a_{0}-a_{5}a_{3}+a_{6}c_{1}-c_{2}a_{4}=0,\\
R_9&:a_{7}c_{1}-a_{5}a_{2}-c_{3}a_{4}+c_{5}a_{0}=0.
\end{align*}
\end{example}
\noindent
This method results in an ideal that describes the same projective variety
\[\mathcal{V}:\bigcap\limits_{i=1}^9{R_i}\backslash N_1\subseteq\mathds{P}^{15}(\mathds{R}),\]
 as in the previous Section. Therefore, we give the $R_i$ as linear combinations of the $Q_i$
\begin{align*}
R_1&=\frac{1}{2}(Q_9-Q_5),&&R_2=\frac{1}{2}(Q_6-Q_2),&R_3=\frac{1}{2}(Q_3-Q_7),\\
R_4&=\frac{1}{2}(Q_8-Q_4),&&R_5=Q_1,&R_6=\frac{1}{2}(Q_9+Q_5),\\
R_7&=\frac{1}{2}(Q_6+Q_2),&&R_8=\frac{1}{2}(Q_3+Q_7),&R_9=\frac{1}{2}(Q_8+Q_4).\\
\end{align*}
The advantage of both presented methods is that we can calculate a Gr\"obner basis for the ideal and apply the theory of ideals to the constructed point models.

\section{Conclusion}
Old and well-known kinematic mappings were unified in one framework by the use of Clifford algebras. Collineations in any kinematic image and the belonging Cayley-Klein space can be derived from the homogeneous Clifford algebra model. We presented the Euclidean spaces of dimension $2$ and $3$ in detail as example. Furthermore, a general method to construct pseudo algebraic varieties in certain projective spaces was presented. These pseudo algebraic varieties are point models for the Spin group of a the homogeneous Clifford algebra model respectively the collineation group of a certain Cayley-Klein space. We performed the construction for the $4$-dimensional Euclidean. Moreover, the general construction enables the examination of new kinematic image spaces for possible Spin groups. Note that this construction can also be done for Pin groups or for any other subgroup of a Clifford group. Due to the fact that projective varieties correspond to ideals, methods of Gr\"obner basis calculus can now be applied to kinematic image spaces.

\section*{Acknowledgement}
This work was supported by the research project "Line Geometry for Lightweight Structures", funded by the DFG (German Research Foundation) as part of the SPP 1542.

\end{JGGarticle}
\end{document}